\documentclass[12pt]{amsart}
\usepackage{latexsym, amsmath, amscd, amssymb, amsthm} 

\newtheorem{te}{Theorem}[section]

 \newtheorem{lm}{Lema}[section] 
 
\begin{document}

\noindent

 \title[ I. The Poincar\'e series]{  Linear locally nilpotent derivations and the classical invariant theory, I:  The Poincar\'e series }

\author{Leonid Bedratyuk}\address{Khmelnitskiy national university, Insituts'ka, 11,  Khmelnitskiy, 29016, Ukraine}

\begin{abstract} 
  By using classical invariant theory approach  a formulas for computation of the Poincar\'e series   of the kernel of   linear locally nilpotent derivations is found. 

\end{abstract}
\maketitle

\section{Introduction} 

  Let $\mathbb{K}$ be a field of characteristic 0. A derivation $\mathcal{D}$
of the polynomial algebra $\mathbb{K}[\mathcal{Z}_n],$ \\ ${\mathcal{X}_n= \{z_1,z_2,\ldots,z_n \}}$
is called a linear  derivation if 
$$
\mathcal{D}(z_i)=\sum_{j=1}^{n} a_{i,j} z_j, a_{i,j} \in \mathbb{K}, i=1,\ldots, n.
$$
If the matrix $A_{\mathcal{D}}:=\{ a_{i,j} \},$ $ i,j=1,\ldots, n $ is  nilpotent one  then the linear derivation is called  a Weitzenb\"ock derivation. The Weitzenb\"ock derivation is a locally nilpotent derivation of  $\mathbb{K}[z_1,z_2,\ldots,z_n].$
Any   Weitzenb\"ock derivation $\mathcal{D}$ is  completely determined by   Jordan normal form of the matrix $A_{\mathcal{D}}.$ Denote by  $\mathcal{D}_{\mathit{\mathbf{d}}},$  $\mathit{\mathbf{d}}:=(d_1,d_2,\ldots, d_s)$   a  Weitzenb\"ok derivation with   Jordan normal form of $A_{\mathcal{\mathcal{D}_{\mathit{\mathbf{d}}}}}$ which consists of  $s$  Jordan blocks of sizes  $d_1+1,$ $d_2+1,$ $\ldots, d_s+1.$ À Weitzenb\"ock derivation which determined by unique Jordan block of size $d+1$ is called the basic Weitzenb\"ock derivation and denoted by $\Delta_d.$

The algebra  
$$
{\rm{ker}}\, D_{\mathit{\mathbf{d}}}=\left\{ f \in \mathbb{K}[\mathcal{Z}_n]| \mathcal{D}_{\mathit{\mathbf{d}}}(f)=0 \right\},
$$
is called the kernel  of the derivation  $D_{\mathit{\mathbf{d}}}.$ 
It is well known that the kernel  $\ker \mathcal{D}_{\mathit{\mathbf{d}}}$  is a finitely generated algebra, see \cite{Wei}--\cite{Tyc}.
 However, it remained an open problem to find a minimal system of homogeneous
generators (or even the cardinality of such a system) of the algebra $\rm{ker}\,D_{\mathit{\mathbf{d}}}$ even  for small sets ${\mathit{\mathbf{d}}}.$

On the other hand, the problem to  describe of the kernel $\ker \mathcal{D}_{\mathit{\mathbf{d}}}$  can be reduced to   an old  problem of the classical invariant theory, namely to the problem to describe  of the algebra of joint covariants of several  binary forms.

It is well known that there is a one-to-one correspondence between  $\mathbb{G}_a$-actions on an affine algebraic variety $V$   and locally nilpotent  $\mathbb{K}$-derivations on its algebra of polynomial functions. Let us identify the algebra  $\mathbb{K}[\mathcal{Z}_n]$  with the algebra  $\mathcal{O}[\mathbb{K}^{n}]$ of polynomial functions of the algebraic variety   $\mathbb{K}^{n}.$    Then, the kernel of the derivation   $\mathcal{D}_{\mathit{\mathbf{d}}}$ coincides with
the invariant ring of the induced via $\exp(t \, \mathcal{D}_{\mathit{\mathbf{d}}})$ action:
$$
\ker \mathcal{D}_{\mathit{\mathbf{d}}}= \mathbb{K}[\mathcal{Z}_n]^{\mathbb{G}_a} \cong \mathcal{O}(\mathbb{K}^n)^{\mathbb{G}_a}.
$$

Now, let $B_{d_1}, B_{d_2},\ldots B_{d_s} $ be the  vector $\mathbb{K}$-spaces of  binary forms of degrees $d_1,d_2,\ldots, d_s$ endowed with the natural action of the group  $SL_2.$ Consider the induced action of the group  $SL_2$ on the algebra of polynomial functions   $\mathcal{O}[B_{\mathbf{d}} \oplus \mathbb{K}^2 ]$   on the vector  space $B_{\mathbf{d}} \oplus \mathbb{K}^2,$ where $$B_{\mathbf{d}}:=B_{d_1}\oplus B_{d_2}\oplus \ldots \oplus B_{d_s}, \dim( B_{\mathbf{d}})=d_1+d_2+\ldots+d_s+s.$$ 
Let  $U_2$ be the maximal unipotent subgroup of the group  $SL_2.$  The application of the Grosshans principle, see  \cite{Gross}, \cite{Pom}  gives

$$
\mathcal{O}[B_{\mathbf{d}} \oplus \mathbb{K}^2 ]^{\,SL_2}\cong \mathcal{O}[B_{\mathbf{d}} ]^{\,U_2}.
$$
Thus
$$
\mathcal{O}[B_{\mathbf{d}} \oplus \mathbb{K}^2 ]^{\,\mathfrak{sl_2}}\cong \mathcal{O}[B_{\mathbf{d}} ]^{\,\mathfrak{u_2}}.
$$
Since $U_2 \cong (\mathbb{K},+)$ and $\mathbb{K}z_1\oplus \mathbb{K}z_2 \oplus \ldots \oplus \mathbb{K}z_n \cong  B_{\mathbf{d}}$  it follows  $$\ker  \mathcal{D}_{\mathit{\mathbf{d}}} \cong \mathcal{O}[B_{\mathbf{d}} \oplus \mathbb{K}^2 ]^{\,\mathfrak{sl_2}}.$$

In the language  of classical invariant theory   the algebra    $\mathcal{C}_{\mathbf{d}}:=\mathcal{O}[B_{\mathbf{d}} \oplus \mathbb{K}^2 ]^{\,\mathfrak{sl_2}}$ is called     the  algebra of joint covariants for $s$  binary forms,  the algebra    ${\mathcal{S}_{\mathbf{d}}:=\mathcal{O}[B_{\mathbf{d}} ]^{\,\mathfrak{u_1}}}$ is called     the  algebra of joint semi-invariants for   binary forms  and the algebra $\mathcal{I}_{\mathbf{d}}:=\mathcal{O}[B_{\mathbf{d}} ]^{\,\mathfrak{sl_2}}$  is called     the  algebra of invariants  for   binary forms of  degrees $d_1,$ $d_2,\ldots, d_s.$ 
The algebras of joint covariants of the binary forms   were an  object of research in the classical invariant theory
of the 19th century.

The reductivity of  $SL_2$  implies that the algebras  $\mathcal{I}_{\mathbf{d}},$  $\mathcal{S}_{\mathbf{d}}\cong \ker \mathcal{D}_{\mathbf{d}},$   are  finitely  generated $\mathbb{Z}$-graded algebras.
The formal power series  $\mathcal{PI}_{\mathbf{d}}, \mathcal{PD}_{\mathbf{d}}= \mathcal{PS}_{\mathbf{d}} \in \mathbb{Z}[[z]],$
$$
\mathcal{PI}_{\mathbf{d}}(z)=\sum_{i=0}^{\infty }\dim((\mathcal{I}_{{\mathbf{d}}})_i) z^i,  \mathcal{PS}_{{\mathbf{d}}}(z) =\sum_{i=0}^{\infty }\dim((\mathcal{S}_{{\mathbf{d}}})_i) z^i,
$$ 
are  called the Poincar\'e series   of the algebras of  joint   invariants and semi-invariants.
 The finitely generation of the algebras  $\mathcal{I}_{\mathbf{d}}$ and $\mathcal{S}_{\mathbf{d}}$ implies  that their  Poincar\'e series are   expansions  of certain  rational functions.  We consider here the problem of
computing efficiently these rational functions. It can  be the first step in describing these algebras.

 Let us recall that the Poincar\'e series of the algebra of covariants for binary form of degree $d$ equals the Poincar\'e series of kernel  of  the basic Weitzenb\"ock derivation $\Delta_d.$
For  the cases  $d\leq 10,$ $d=12$ the  Poincar\'e series of the algebra of  invariants and covariants for the  binary $d$-form   were calculated by Sylvester and  Franklin, see  \cite{SF}, \cite{Sylv-12}. To do so, they used the Sylvester-Cayley formula for the dimension of graded subspaces.
 In  \cite{Ono} the Poincar\'e  series for  $\Delta_5$ was rediscovered. Springer \cite{SP} derived  the  formula for computing the Poincar\'e  series of the algebras of invariants of the binary $d$-forms.This formula has been used by Brouwer and Cohen  \cite{BC} for  the Poincar\'e  series calculations in the cases  $d\leq 17$ and also by Littelmann and Procesi  \cite{LP} for even  $d\leq 36.$ For the case $d\leq 30$  in   \cite{BI} the explicit form of the Poincare series is given.

In  \cite{BC1}, \cite{BC2}  we have found Sylvester-Cayley  type and Springer type  formulas for the basic  derivation $\Delta_d$  and for the derivation   $\mathcal{D}_{{\mathbf{d}}}$ for ${\mathbf{d}}=(d_1,d_2).$ Also, for those derivations the Poincar\'e series was found for $d,d_1,d_2 \leq 30.$
Relatively recently, in \cite{Dr_G}  the formula for computing the Poincare series of  the Weitzenb\"ock derivation  $\mathcal{D}_{{\mathbf{d}}}$ for arbitrary ${\mathbf{d}}$ was  announced.

 In this paper we  have  given Sylvester-Cayley  type formulas  for calculation of $\dim((\mathcal{I}_{{\mathbf{d}}})_i,$  $\dim (\ker  \mathcal{D}_{\mathbf{d}})_i$ and, Springer-type formulas  for calculation of $\mathcal{PI}_{\mathbf{d}}(z),$ $\mathcal{PD}_{\mathbf{d}}(z)= \mathcal{PS}_{\mathbf{d}}(z)$ for arbitrary $\mathbf{d}.$ Also, for the cases $\mathbf{d}=(1,1,\ldots,1),$ $\mathbf{d}=(2,2,\ldots,2)$ the explicit formulas for $\mathcal{PI}_{\mathbf{d}}(z),$ $\mathcal{PD}_{\mathbf{d}}(z)$  are  given.

\section{Sylvester-Cayley type formula  for the kernel }

 To begin with, we give a proof of the Sylvester-Cayley type  formula for ðîçì³ðíîñò³ ãðàäóéîâàíèõ ï³äïðîñòîð³â îô  the kernel of   Weitzenb\"ock derivations $\mathcal{D}_{\mathit{\mathbf{d}}},$ $ \mathbf{d}: =(d_1, d_2, \ldots ,d_s )$.
%===========================================================================

Let us consider the polynomial algebra    $\mathbb{K}[X_{\mathbf{d}}]$ generated by the set of variables   $$
X_{\mathbf{d}}:=\left \{x^{(1)}_{0},x^{(1)}_{1},\ldots, x^{(1)}_{d_1},x^{(2)}_{0},x^{(1)}_{2},\ldots, x^{(2)}_{d_2}, \ldots x^{(s)}_{0},x^{(s)}_{1},\ldots, x^{(s)}_{d_s} \right \}.
$$

Define on    $\mathbb{K}[X_{\mathbf{d}}]$ the linear nilpotent derivation     $\mathcal{D}_{\mathit{\mathbf{d}}},$ $ \mathbf{d}: =(d_1, d_2, \ldots ,d_s )$     by 
$$
\mathcal{D}_{\mathit{\mathbf{d}}}(x_i^{(k)})=i \, x_{i-1}^{(k)}, k=1,\ldots,s.
$$
Also, define on   $\mathbb{K}[X_{\mathbf{d}}]$  two linear derivations    $\mathcal{D}_{\mathit{\mathbf{d}}}^*$ and $\mathcal{E}_{\mathit{\mathbf{d}}},$ by 
$$
\mathcal{D}^*_{\mathit{\mathbf{d}}}(x_i^{(k)})=(d_k-i) \, x_{i+1}^{(k)}, \mathcal{E}_{\mathit{\mathbf{d}}}(x_i^{(k)})=(d_k-2\,i) \, x_{i}^{(k)}, k=1,\ldots,s.
$$
The linear locally nilpotent derivation $\mathcal{D}^*_{\mathit{\mathbf{d}}}$ is said to be  {\it the dual } derivation with respect to the derivation $\mathcal{D}_{\mathit{\mathbf{d}}}.$

By direct calculation we get 
$$
\left[\mathcal{D}_{\mathit{\mathbf{d}}},\mathcal{D}_{\mathit{\mathbf{d}}}^*\right]( x_{i}^{(k)})=\mathcal{D}_{\mathit{\mathbf{d}}}\left(\mathcal{D}^*_{\mathit{\mathbf{d}}}\left(x_i^{(k)}\right)\right)-\mathcal{D}^*_{\mathit{\mathbf{d}}}\left(\mathcal{D}_{\mathit{\mathbf{d}}}\left(x_i^{(k)}\right)\right)=(d_k-2i)\, x_{i}^{(k)}=\mathcal{E}_{\mathit{\mathbf{d}}}( x_{i}^{(k)}).
$$
In the same way we get  $[\mathcal{D}_{\mathit{\mathbf{d}}},\mathcal{E}_{\mathit{\mathbf{d}}}]=-2 \mathcal{D}_{\mathit{\mathbf{d}}}$ and $[\mathcal{D}^*_{\mathit{\mathbf{d}}},\mathcal{E}_{\mathit{\mathbf{d}}}]=2\mathcal{D}^*_{\mathit{\mathbf{d}}}.$ Therefore, the polynomial algebra $\mathbb{K}[X_{\mathbf{d}}]$ considered as a vector space   becomes  a  $\mathfrak{sl_
{2}}$--module. The basis elements    $ \left( \begin{array}{ll}  0\, 1 \\ 0\,0 \end{array} \right),$ $ \left( \begin{array}{ll}  0\, 0 \\ 1\,0 \end{array} \right)$, $ \left( \begin{array}{ll}  1 &  \phantom{-}0 \\  0 &-1 \end{array} \right)$ of the algebra    $\mathfrak{sl_{2}}$ act on    $\mathbb{K}[X_{\mathbf{d}}]$   as follows: 
$$
 \left( \begin{array}{ll}  0\, 1 \\ 0\,0 \end{array} \right). f=\mathcal{D}_{\mathit{\mathbf{d}}}\left(f\right), \left( \begin{array}{ll}  0\, 0 \\ 1\,0 \end{array} \right). f=\mathcal{D}^*_{\mathit{\mathbf{d}}}\left(f\right), \left( \begin{array}{ll}  1 &  \phantom{-}0 \\  0 &-1 \end{array} \right). f=\mathcal{E}_{\mathit{\mathbf{d}}}\left(f\right),
$$
for any $f \in \mathbb{K}[X_{\mathbf{d}}].$

Let   $\mathfrak{u}_{2}=\mathbb{K}[X_{\mathbf{d}}]$ be  the maximal unipotent subalgebra of $\mathfrak{sl}_{2}.$ As above, let us identify   the algebras $\mathcal{I}_{\mathbf{d}},$
 $\mathcal{S}_{\mathbf{d}},$
$$
\begin{array}{l}
\displaystyle \mathcal{I}_{\mathbf{d}}:= \displaystyle{\mathbb{K}[X_{\mathbf{d}}]^{\mathfrak{sl_{2}}}}=\{ v \in \mathbb{K}[X_{\mathbf{d}}]|  \mathcal{D}_{\mathbf{d}}(v)= \mathcal{D}^*_{\mathbf{d}}(v)=0 \},\\ 
\displaystyle \mathcal{S}_{\mathbf{d}}:=\ker \mathcal{D}_{\mathbf{d}}= \displaystyle{ \mathbb{K}[X_{\mathbf{d}}]^{\mathfrak{u_{2}}}}=\{ v \in \mathbb{K}[X_{\mathbf{d}}]|  \mathcal{D}_{\mathbf{d}}(v)=0 \},
\end{array}
$$
with   the algebras of  {joint invariants} and  {\it joint semi-invariants}  of the binary forms of the degrees   $d_1,$ $d_2,$  $\ldots,$ $d_s.$ For any element $v \in \mathcal{S}_{\mathbf{d}}$ a natural number $m$ is called the  order of the element $v$ if the number $r$ is the smallest natural number such that \begin{equation*}(\mathcal{D}^*_{\mathit{\mathbf{d}}})^r(v) \ne 0, (\mathcal{D}^*_{\mathit{\mathbf{d}}})^{r+1}(v) = 0.\end{equation*}
It is  clear that any semi-invariant   of  order $r$ is the highest weight vector  for an  irreducible $\mathfrak{sl_{2}}$-module   of the dimension $r+1$ in $\mathbb{K}[X_{\mathbf{d}}].$
 
The algebra simultaneous covariants is isomorphic to the algebra of  simultaneous semi-invariants. Therefore, it is  enough to compute the Poincar\'e  series of the algebra $\mathcal{S}_{\mathbf{d}}.$

The algebras    $\mathbb{K}[X_{\mathbf{d}}],$ $ \mathcal{I}_{\mathbf{d}},$ $ \mathcal{S}_{\mathbf{d}}$ are  graded algebras:  
$$
\begin{array}{l}
\mathbb{K}[X_{\mathbf{d}}]=(\mathbb{K}[X_{\mathbf{d}}])_0+(\mathbb{K}[X_{\mathbf{d}}])_1+\cdots +(\mathbb{K}[X_{\mathbf{d}}])_m+\cdots,\\
 \mathcal{I}_{\mathbf{d}}=( \mathcal{I}_{\mathbf{d}})_0+( \mathcal{I}_{\mathbf{d}})_1+\cdots+( \mathcal{I}_{\mathbf{d}})_m+\cdots,\\
\mathcal{S}_{\mathbf{d}}=( \mathcal{S}_{\mathbf{d}})_0+( \mathcal{S}_{\mathbf{d}})_1+\cdots+( \mathcal{S}_{\mathbf{d}})_m+\cdots.
\end{array}
$$
and each   $(\mathbb{K}[X_{\mathbf{d}}])_m$ is the complete reducible
 representation of the Lie algebra  $\mathfrak{sl_{2}}.$

Let  $V_k$ be  the standard irreducible    $\mathfrak{sl_{2}}$-module, $\dim V_k=k+1.$ Then, the following primary decomposition  holds
$$
(\mathbb{K}[X_{\mathbf{d}}])_m \cong \gamma_m({\mathbf{d}};0) V_0+\gamma_m({\mathbf{d}};1) V_1+ \cdots +\gamma_m({\mathbf{d}};m \cdot d^*) V_{m\,d^{*}},  \eqno{(1)}
$$
here  $d^*:=\max(d_1,d_2,\ldots d_s)$ and $\gamma_m({\mathbf{d}};k)$ is  the  multiplicity of the representation  $V_k$  in the decomposition of  $(\mathbb{K}[X_{\mathbf{d}}])_m.$ On the other hand, the multiplicity  $\gamma_m({\mathbf{d}};k)$  of the  representation  $V_k$ is  equal to the number of linearly independent homogeneous simultaneous semi-invariants of  the degree $m$   and the order $k.$
In particular, the number of linearly  independent simultaneous invariants of degree  $m$ is  equal to  $\gamma_m({\mathbf{d}};0).$  These arguments prove 
\begin{lm}
$$
\begin{array}{ll}
(i) & \dim (\mathcal{I}_{{\mathbf{d}}})_m=\gamma_m({\mathbf{d}};0),\\
(ii) & \dim (S_{\mathbf{d}})_m=\gamma_m({\mathbf{d}};0)+\gamma_m({\mathbf{d}};1) + \cdots +\gamma_m({\mathbf{d}};m\,d^*).
\end{array}
$$
\end{lm}
Let us recall some general facts about the representation theory of the Lie algebra  $\mathfrak{sl_{2}}.$

The set of weights  of a representation  $W$ denote by  $\Lambda_{W},$  in particular, $\Lambda_{V_d}=\{-d, -d+2, \ldots, d \}.$ 
 A  formal sum 
$$
{\rm Char}(W)=\sum_{\lambda \in \Lambda_{W}} n_W(\lambda) q^{\lambda},
$$
is called the character   of a representation  $W,$  
here   $n_W(\lambda)$ denotes the   multiplicity  of the weight $\lambda \in \Lambda_{W}.$
Since, a  multiplicity of any weight of the irreducible representation $V_d$  is  equal to 1,   we have  
$$
{\rm Char}(V_d)=q^{-d}+q^{-d+2}+\cdots+q^{d}.
$$
Let us consider the  $s$  sets of variables  $x^{(1)}_{0},x^{(1)}_{1},\ldots, x^{(1)}_{d_1},$ $x^{(2)}_{0},x^{(1)}_{2},\ldots, x^{(2)}_{d_2},$ $\ldots,$ $x^{(s)}_{0},x^{(s)}_{1},\ldots, x^{(s)}_{d_s}.$
The character  $ {\rm Char}\left((\mathbb{K}[X_{\mathbf{d}}])_m\right)$ of the representation  $(\mathbb{K}[X_{\mathbf{d}}])_m,$ see  \cite{FH}, equals   $$H_m(q^{-d_1},q^{-d_1+2},\ldots,q^{d_1},q^{-d_2},q^{-d_2+2},\ldots,q^{d_2},\ldots,q^{-d_s},q^{-d_s+2},\ldots,q^{d_s}),$$   where  $H_m(x^{(1)}_{0},x^{(1)}_{1},\ldots, x^{(1)}_{d_1},\ldots,x^{(s)}_{0},x^{(s)}_{1},\ldots, x^{(s)}_{d_s})$ is  the complete symmetrical function     
$$
\begin{array}{l}
\displaystyle H_m(x^{(1)}_{0},x^{(1)}_{1},\ldots, x^{(1)}_{d_1},\ldots,x^{(s)}_{0},x^{(s)}_{1},\ldots, x^{(s)}_{d_s})=\\
\\
\displaystyle =\sum_{|\alpha^{(1)}|+\ldots+|\alpha^{(s)}|=m} (x^{(1)}_{0})^{\alpha_0^{(1)}}(x^{(1)}_{1})^{\alpha_1^{(1)}}\ldots (x^{(1)}_{d_1})^{\alpha_{d_1}^{(1)}}\cdots (x^{(s)}_{0})^{\alpha_0^{(s)}}(x^{(s)}_{1})^{\alpha_1^{(s)}} \ldots (x^{(1)}_{d_s})^{\alpha_{d_1}^{(s)}},
\end{array}
$$
where $\displaystyle |\alpha^{(k)}|:=\sum_{i=0}^{d_i}\alpha_i^{(k)}.$

By replacing   $x_i^{(k)}=q^{d_k-2\,i},$   we  obtain the specialized expression for the character   $(\mathbb{K}[X_{\mathbf{d}}])_m,$ namely 
$$
\begin{array}{c}
\displaystyle {\rm Char}((\mathbb{K}[X_{\mathbf{d}}])_m)= \\
\\
\displaystyle =\sum_{|\alpha^{(1)}|+\ldots+|\alpha^{(s)}|=n} (q^{d_1})^{\alpha_0^{(1)}} (q^{d_1-2\cdot 1})^{\alpha_1^{(1)}} \ldots (q^{-d_1})^{\alpha_{d_1}^{(1)}} \ldots  (q^{d_s})^{\alpha_0^{(s)}} (q^{d_s-2\cdot 1})^{\alpha_1^{(s)}} \ldots (q^{-d_s})^{\alpha_{d_s}^{(s)}} =
\\
\\
\displaystyle = \sum_{|\alpha^{(1)}|+\ldots+|\alpha^{(s)}|=n} q^{d_1|\alpha^{(1)}|+\ldots+d_s|\alpha^{(s)}| +\left(\alpha_1^{(1)}+2\alpha_2^{(1)}+\cdots + d_1\, \alpha_{d_1}^{(1)}\right)+\ldots+\left(\alpha_1^{(s)}+2\alpha_2^{(s)}+\cdots + d_s\, \alpha_{d_s}^{(s)}\right)}=\\
\\
\displaystyle =\sum_{i=-m\, d^*}^{m\,d^*} \omega_n({\mathbf{d}};i) q^{i},
\end{array}
$$
here    $\omega_m({\mathbf{d}};i)$  is the number of nonnegative integer solutions of the following  system of equations:
$$
\left \{
\begin{array}{l}
d_1|\alpha^{(1)}|+\ldots+d_s|\alpha^{(s)}| +\left(\alpha_1^{(1)}+2\alpha_2^{(1)}+\cdots + d_1\, \alpha_{d_1}^{(1)}\right)+\\
+\ldots+\left(\alpha_1^{(s)}+2\alpha_2^{(s)}+\cdots + d_s\, \alpha_{d_s}^{(s)}\right)=i \\
\\
|\alpha^{(1)}|+\ldots+|\alpha^{(s)}|=m.
\end{array}
\right. 
 \eqno{(2)}
$$

We can summarize what we have shown so far in  
\begin{te} 
$$
\begin{array}{ll}
(i) & \dim (\mathcal{I}_{{\mathbf{d}}})_m=\omega_m({\mathbf{d}};0)-\omega_m({\mathbf{d}};2),\\
&\\
(ii) & \dim (\mathcal{S}_{{\mathbf{d}}})_m=\omega_m({\mathbf{d}};0)+\omega_m({\mathbf{d}};1).
\end{array}
$$

\end{te}
\begin{proof}
\noindent
 $(i)$
The zero weight appears  once  in any  representation $V_k,$  for even $k$,  therefore
$$
\omega_m({\mathbf{d}};0)=\gamma_m({\mathbf{d}};0)+\gamma_m({\mathbf{d}};2)  +\gamma_m({\mathbf{d}};4)+\ldots
$$
The weight   $2$ appears  once in any  representation   $V_k,$  for even $k>0$,  therefore
$$
\omega_m({\mathbf{d}};2)=\gamma_m({\mathbf{d}};2)+\gamma_m({\mathbf{d}};4)  +\gamma_m({\mathbf{d}};6)+\ldots
$$
Taking into account Lemma  2.1,  we obtain
$$
\omega_m({\mathbf{d}};0)-\omega_m({\mathbf{d}};2)=\gamma_m({\mathbf{d}};0)= \dim (\mathcal{I}_{{\mathbf{d}}})_m.
$$

\noindent
$(ii)$
The  weight $1$ appears  once  in any  representation $V_k,$  for odd $k$,  therefore
 $$
\omega_m({\mathbf{d}};1)=\gamma_m({\mathbf{d}};1)+\gamma_m({\mathbf{d}};3)  +\gamma_m({\mathbf{d}};5)+\ldots
$$
Thus,
$$
\begin{array}{l}
\displaystyle \omega_m({\mathbf{d}};0)+\omega_m({\mathbf{d}};1)=\\ \\
\displaystyle =\gamma_m({\mathbf{d}};0)+\gamma_m({\mathbf{d}};1)  +\gamma_m({\mathbf{d}};2)+\ldots
+\gamma_m({\mathbf{d}};n\,d)=\\\\
\displaystyle =\dim (S_{{\mathbf{d}}})_m.
\end{array}
$$
\end{proof}

Simplify the system  $(2)$  to 
$$
\left \{
\begin{array}{l}
d_1\alpha_0^{(1)}+(d_1-2)\alpha_1^{(1)}+(d_1-4)\alpha_2^{(1)}+\cdots + (-d_1)\, \alpha_{d_1}^{(1)}+\cdots +\\ \\+d_s\alpha_0^{(s)}+(d_s-2)\alpha_1^{(s)}+(d_s-4)\alpha_2^{(s)}+\cdots + (-d_s)\, \alpha_{d_s}^{(s)}=i, \\
\\
\alpha_0^{(1)}+\alpha_1^{(1)}+\cdots +\alpha_{d_1}^{(1)}+\cdots +\alpha_0^{(s)}+\alpha_1^{(s)}+\cdots +\alpha_{d_1}^{(s)}=n.
\end{array}
\right. 
$$
 It well-known  that  the number
 $\omega_m({\mathbf{d}};i)$ of   non-negative integer solutions of the above  system
is equal to the coefficient of  $\displaystyle t^m z^i $ of the  expansion   of the  series
$$
\begin{array}{l}
f_{{\mathbf{d}}}(t,z)=\\
\\
=\displaystyle \frac{1}{(1-t z^{d_1})(1-t\,z^{d_1-2})\ldots (1-t\,z^{-d_1})\cdots (1-t z^{d_s})(1-t\,z^{d_s-2})\ldots (1-t\,z^{-d_s})}.
\end{array}
$$
Denote it in such a way:  $\omega_m({\mathbf{d}};i):=\left[  t^m z^i\right](f_{{\mathbf{d}}}(t,z)).$ 
Observe that $ f_{{\mathbf{d}}}(t,z)=f_{{\mathbf{d}}}(t,z^{-1}).$

The  following statement holds
\begin{te} 
$$
\begin{array}{ll}
(i) & \dim (I_{{\mathbf{d}}})_m=[t^m ](1-z^2)f_{{\mathbf{d}}}(t,z),\\
&\\
(ii) & \dim (\mathcal{S}_{{\mathbf{d}}})_m=[t^m ](1+z)f_{{\mathbf{d}}}(t,z).
\end{array}
$$
\end{te}
\begin{proof}
Taking into account the formal property $[x^{i-k}]f(x)=[x^{i}](x^k f(x)),$ we  get
$$
\begin{array}{l}
 \dim (I_{{\mathbf{d}}})_m=\omega_m({\mathbf{d}};0)-\omega_m({\mathbf{d}};2)=[t^m]f_{{\mathbf{d}}}(t,z)-[t^m\,z^2]f_{{\mathbf{d}}}(t,z)=\\
\\
=[t^m]f_{{\mathbf{d}}}(t,z)-[t^m]z^{-2} f_{{\mathbf{d}}}(t,z)=[t^m]f_{{\mathbf{d}}}(t,z)-[t^m]z^{2} f_{{\mathbf{d}}}(t,z^{-1})=\\
\\
=[t^m](1-z^{2}) f_{{\mathbf{d}}}(t,z).
\end{array}
$$

In the same way
$$
\begin{array}{l}
 \dim (S_{{\mathbf{d}}})_m=\omega_m({\mathbf{d}};0)+\omega_m({\mathbf{d}};1)=[t^m]f_{{\mathbf{d}}}(t,z)+[t^m\,z]f_{{\mathbf{d}}}(t,z)=\\
\\
=[t^m]f_{{\mathbf{d}}}(t,z)+[t^m]z^{-1} f_{{\mathbf{d}}}(t,z)=[t^m](1+z) f_{{\mathbf{d}}}(t,z).
\end{array}
$$
\end{proof}

It is easy to see that the dimensions   $\dim (I_{{\mathbf{d}}})_m,$ $\dim (S_{{\mathbf{d}}})_m$ alow the folloving representations:
$$
\begin{array}{ll}
 & \displaystyle \dim (I_{{\mathbf{d}}})_m=[t^m]\frac{1}{2\pi i} \oint_{|z|=1}(1-z^2) f_{{\mathbf{d}}}(t,z) \frac{dz}{z},\\
&\\
 & \displaystyle \dim (\mathcal{S}_{{\mathbf{d}}})_m=[t^m]\frac{1}{2\pi i} \oint_{|z|=1}(1+z) f_{{\mathbf{d}}}(t,z) \frac{dz}{z}.
\end{array}
$$

\section{Springer type  formulas for the Poincar\'e  series }
%%%%%%%%%%%%%%%%%%%%%%%%%%%%%%%%%%%%%%%%%%%%%%%%%%%%%%
 Let us prove  a Springer type  formulas for the Poincar\'e  series  $\mathcal{PI}_{{\mathbf{d}}}(z),$ $ \mathcal{PS}_{{\mathbf{d}}}(z)= \mathcal{PD}_{{\mathbf{d}}}(z)$  of the  algebras  simultaneous invariants and  semi-invariants of  two binary forms. 

Consider the $\mathbb{C}$-algebra $\mathbb{C}[[t,z]]$   of  the formal   power series.
For an arbitrary   $m,n \in \mathbb{Z^+}$  define  $\mathbb{C}$-linear function
$$ \Psi_{m,n}:\mathbb{C}[[t,z]] \to \mathbb{C}[[z]],$$  in the  following  way:
$$
\Psi_{m,n}\left(\sum_{i,j=0}^{\infty} a_{i,j}\, t^i z^j\right)=\sum_{i=0}^{\infty} a_{i m,i n} z^i.
$$
Denote by $\varphi_n$  the restriction of 
$\Psi_{m,n}$ to  $\mathbb{C}[[z]],$ namely 

$$
\varphi_{n}\left(\sum_{i=0}^{\infty} a_{i}z^i \right)=\sum_{i=0}^{\infty} a_{i n} z^i. 
$$
There is an effective algorithm of   calculation  for the function $ \varphi_n, $ see  
 \cite{BC1}.
In  some  cases   calculation of the functions $ \Psi$  can  be reduced   to  calculation of  the functions    $\varphi$.  The following  statements hold:

\begin{lm} For   $R(z) \in \mathbb{C}[[z]]$  and for   $ m, n, k  \in \mathbb{N}$    we have:

$$
\begin{array}{ll}

 &  \displaystyle  \Psi_{1,n}\left( \frac{R(z)}{(1-t z^k)^m} \right)=\left \{ \begin{array}{l} \displaystyle \frac{1}{(m-1)!}\frac{d^{m-1}(z^{m-1}\, \varphi_{n-k}(R(z)))}{dz^{m-1}}, n > k; \\ \\ \displaystyle \frac{R(0)}{(1-z)^m},n=k;
 \\ \\ R(0),  \text{  if  } k>n. \end{array} \right. 
\end{array}
$$
\end{lm}
\begin{proof}
 Let  $R(z)=\sum_{j=0}^{\infty} r_{j} z^j.$  Observe, that  
$$
\frac{1}{(1-x)^m}=\frac{1}{(m-1)!}\left[ \frac{1}{1-x}\right]^{(m-1)}_x=\sum_{i=0}^{\infty} { s+m-1 \choose m-1} x^s.
$$
Then for   $n>k$  we have 
$$
\begin{array}{l}
\displaystyle \Psi_{1,n}\left( \frac{R(z)}{(1-t z^k)^m} \right)=\Psi_{1,n}\Big( \sum_{j,s\geq 0}  { s+m-1 \choose m-1}\,r_j z^j (t z^k)^s\Big)=\\ \\ 
\displaystyle =\Psi_{1,n}\Big(\sum_{s\geq 0}  { s+m-1 \choose m-1}r_{s(n-k)}\, (t z^n)^s \Big){=}\sum_{s \geq 0}   { s+m-1 \choose m-1}\, r_{s(n-k)} z^s.
\end{array}
$$
On other hand
 $$
\begin{array}{l}
\displaystyle \frac{1}{(m-1)!}\left(z^{m-1} \varphi_{n-k}(R(z))\right)^{(m-1)}_z=\frac{1}{(m-1)!} \left(\sum_{s=0}^{\infty} r_{s(n-k)} z^{m+s-1}\right)^{(m-1)}_z=\\
\\
\displaystyle = \frac{1}{(m-1)!}\sum_{s \geq 0} (s+m-1)(s+m-2)\ldots (s+1) r_{s(n-k)} z^s=\sum_{s \geq 0}   { s+m-1 \choose m-1}\, r_{s(n-k)} z^s.

\end{array}
$$
This proves the case $n>k.$

Taking into account the formal property $$ \Psi_{1,n}\left(F(tz^n)\, H(t,z)\right)=F(z)  \Psi_{1,n}(H(t,z)),  F(z),H(t,n) \in \mathbb{C}[[t,z]],$$
for the case $n=k$  we  have 
$$
 \Psi_{1,n}\left( \frac{R(z)}{(1-t z^n)^m} \right)= \frac{1}{(1-z)^m}  \Psi_{1,n}\left(R(z)\right)= \frac{R(0)}{(1-z)^m}.
$$
To prove the case  $n<k,$ observe that, the equation $k s+j=n s$  for $n<k$ and $j, s \geq 0$  has only one trivial  solution $j=s=0$. We have
$$
\displaystyle \Psi_{1,n}\left( \frac{R(z)}{1-t z^k} \right)=\Psi_{1,n}\left( \sum_{j,s\geq 0} r_j  z^j (t z^k)^s\right)
{=}\Psi_{1,n}\left( \sum_{j,s\geq 0} r_j t^s  z^{k s+j}\right)=r_0=R(0).
$$

 \end{proof}

The main idea of the calculations of the paper  is that 
the  Poincar\'e series $ \mathcal{PI}_{{\mathbf{d}}}(z),$ $\mathcal{PS}_{{\mathbf{d}}}(z)$  can be expressed  in terms of functions $ \Psi.$ The following simple but important statement  holds:

\begin{lm}  Let   $d^*:=\max({\mathbf{d}}).$ Then 
$$
\begin{array}{ll}
(i) & \mathcal{PI}_{{\mathbf{d}}}(z)=\Psi_{1,d^*}\left((1-z^2)f_{{\mathbf{d}}}(tz^{d^*},z)\right),\\
\\
(ii) & \mathcal{PS}_{{\mathbf{d}}}(z)=\Psi_{1,d^*}\left((1+z)f_{{\mathbf{d}}}(tz^{d^*},z)\right).\\

\end{array}
$$

\end{lm}
\begin{proof} Theorem  2   states  that  $\dim(\mathcal{I}_{{\mathbf{d}}})_n=[t ^n](1-z^2)f_{{\mathbf{d}}}(t,z).$ 
Then 
$$
\begin{array}{l}
\displaystyle  \mathcal{PI}_{{\mathbf{d}}}(z) = \sum_{n=0}^{\infty}  \dim(I_{{\mathbf{d}}})_n z^n=\sum_{n=0}^{\infty} \bigl([t ^n](1-z^2)f_{{\mathbf{d}}}(t,z)\bigr)z^n{=}\\
\\
\displaystyle =\sum_{n=0}^{\infty} \bigl([(tz^{d^*}) ^n](1-z^2)f_{{\mathbf{d}}}(tz^{d^*},z)\bigr)z^n{=} \Psi_{1,d^*}\left((1-z^2)f_{{\mathbf{d}}}(tz^d,z)\right).
\end{array}
$$
Similarly, we prove   the  statement $(ii).$ 

We  replaced  $t$  with   $tz^{d^*}$  to avoid  of a negative powers  of $z$  in the denominator of the function $f_{{\mathbf{d}}}(t,z).$
\end{proof}

Write  the function $f_{{\mathbf{d}}}(t,z)$ in the following way 
$$
f_{{\mathbf{d}}}(t,z)=\frac{1}{\prod_{k=1}^s (tz^{- \,d_k},z^2)_{d_k+1}},
$$
here $(a,q)_n=(1-a) (1-a\,q)\cdots (1-a\,q^{n-1})$ denotes the $q$-shifted factorial.

The above lemma implies the following representations of the Poincar\'e series via the contour integrals:
\begin{lm} 
$$
\begin{array}{ll}
(i) & \displaystyle  \mathcal{PI}_{{\mathbf{d}}}(t)=\frac{1}{2\pi i} \oint_{|z|=1} \frac{ 1-z^2}{\prod_{k=1}^s (tz^{-\, d_k},z^2)_{d_k+1}} \frac{dz}{z}

,\\
\\
(ii) & \displaystyle \mathcal{PS}_{{\mathbf{d}}}(t)=\frac{1}{2\pi i} \oint_{|z|=1} \frac{ 1+z}{\prod_{k=1}^s (tz^{ -\, d_k},z^2)_{d_k+1}} \frac{dz}{z}.
\\
\end{array}
$$
\end{lm}
\begin{proof}
We have 
$$
\begin{array}{l}
\displaystyle  \mathcal{PS}_{{\mathbf{d}}}(t) = \sum_{n=0}^{\infty}  \dim(I_{{\mathbf{d}}})_n t^n=\sum_{n=0}^{\infty} \bigl([t ^n](1+z)f_{{\mathbf{d}}}(t,z)\bigr)t^n{=}\\
\\
\displaystyle = \sum_{n=0}^{\infty}\left([t ^n] \frac{1}{2\pi i} \oint_{|z|=1}(1+z) f_{{\mathbf{d}}}(t,z) \frac{dz}{z} \right)
t^n{=}
 \frac{1}{2\pi i} \oint_{|z|=1}(1+z) f_{{\mathbf{d}}}(t,z) \frac{dz}{z}.
\end{array}
$$

Similarly we get the Poincar\'e series  $\mathcal{PI}_{{\mathbf{d}}}(t).$
\end{proof}

Note that   the Molien-Weyl integral formula  for the Poincar\'e series  $\mathcal{P}_d(t)$ of the algebra of  invariants of   binary $d$-form  can be reduced to the following formula
$$
\mathcal{P}_d(t)=\frac{1}{2\pi i} \oint_{|z|=1} \frac{ 1-z^2}{(1-tz^d)(1-tz^{d-2})\ldots (1-tz^{-d})} \frac{dz}{z}=\frac{1}{2\pi i} \oint_{|z|=1} \frac{ 1-z^2}{(tz^{-\, d},z^2)_{d+1}} \frac{dz}{z}.
$$
 see \cite{DerK}, p. 183.  An ingenious way to calculate such  integrals proposed in \cite{DZ}. 

%%%%%%%%%% ----------------d_2>d_1----------------%%%%%%%%%%%%%%%%

After simplification we  can write  $f_{{\mathbf{d}}}(tz^{d^*},z)$ in the following way
$$
f_{{\mathbf{d}}}(tz^{d^*},z)=\left((1-t)^{\beta_0} (1-t z)^{\beta_1} (1-t z^2)^{\beta_2} \ldots (1-t z^{2\,d^*})^{\beta_{2\,d^*}}\right)^{-1},
$$
for some integer  $\beta_0, \ldots \beta_{d^*}.$  For example
$$
f_{(1,2,4)}(tz^{4},z)={\frac {1}{\left( 1-t \right) \left( 1-t{z}^{2} \right) ^{2}\left( 1-t z^3 \right) \left( 1-t{z}^{4} \right) ^
{2} \left( 1-t z^5 \right)\left( 1-t{z}^{6} \right) ^{2}   \left( 1-t{z}^{
8} \right) }}.
$$

It implies the following partial fraction decomposition of 
 $f_{{\mathbf{d}}}(tz^{d^*},z):$ 
$$
f_{{\mathbf{d}}}(tz^{d^*},z)=\sum_{i=0}^{2\,d^*} \sum_{k=1}^{\beta_i} \frac{A_{i,k}(z)}{(1-tz^i)^k}, 
$$
for some polynomials $A_{i,k}(z).$

By direct calculations we obtain
$$
\begin{array}{l}
\displaystyle  A_{i,k}(z)=\frac{(-1)^{\beta_i-k}}{(\beta_i-k)!\,(z^{i})^{{\beta_i-k}}}\lim_{t \to z^{-i}} \frac{\partial^{{\beta_i-k}}}{\partial t^{{\beta_i-k}}}\left(f_{d}(tz^{d^*},z)(1-t z^{i})^{\beta_i} \right).
\end{array}
$$

Now we  can present    Springer type  formulas  for the Poincar\'e  series  $\mathcal{PI}_{{\mathbf{d}}}(z)$ and $\mathcal{PS}_{{\mathbf{d}}}(z).$
\begin{te} 
$$
\begin{array}{l}
\displaystyle \mathcal{PI}_{{\mathbf{d}}}(z)=\sum_{i=0}^{d^*} \sum_{k=1}^{\beta_i} \frac{1}{(k-1)!} \frac{d^{k-1}\left(  z^{k-1} \varphi_{d^*-k}((1-z^2)\,A_{i,k}(z))\right)}{dz^{k-1}},\\
\\
\displaystyle \mathcal{PS}_{{\mathbf{d}}}(z)=\sum_{i=0}^{d^*} \sum_{k=1}^{\beta_i} \frac{1}{(k-1)!} \frac{d^{k-1}\left(  z^{k-1} \varphi_{d^*-k}((1+z)\,A_{i,k}(z))\right)}{dz^{k-1}}.
\end{array}
$$
\end{te}
\begin{proof}
Taking into account  Lemma 3.1 and linearity of the map $\Psi$  we get 
$$
\begin{array}{l}
\displaystyle \mathcal{PS}_{{\mathbf{d}}}(z)=\Psi_{1,d^*}\left((1+z) f_{{\mathbf{d}}}(tz^{d^*},z) \right)=\Psi_{1,d^*}\left(\sum_{i=0}^{2\,d^*} \sum_{k=1}^{\beta_i} \frac{(1+z)A_{i,k}(z)}{(1-tz^i)^k} \right)=\\
\\
\displaystyle = \sum_{i=0}^{d^*} \sum_{k=1}^{\beta_i} \frac{1}{(k-1)!} \frac{d^{k-1}\left(  z^{k-1} \varphi_{d^*-k}((1+z)\,A_{i,k}(z))\right)}{dz^{k-1}}.
\end{array}
$$

The case  $ \mathcal{PI}_{{\mathbf{d}}}(z)$ can be considered similarly.
\end{proof}
Note, the  Poincar\'e series  $\mathcal{PI}_{d}(z)$  and  $\mathcal{PC}_{d}(z)$  of the algebras of invariants and  covariants of   binary $d$-form  equal
$$
\begin{array}{l}
\displaystyle \mathcal{PI}_{d}(z)=\sum_{0\leq k <d/2} \varphi_{d-2\,k} \left( \frac{(-1)^k z^{k(k+1)} (1-z^2)}{(z^2,z^2)_k\,(z^2,z^2)_{d-k}} \right),
\end{array}
$$
$$
\begin{array}{l}
\displaystyle \mathcal{PC}_{d}(z)=\sum_{0\leq k <d/2} \varphi_{d-2\,k} \left( \frac{(-1)^k z^{k(k+1)} (1+z)}{(z^2,z^2)_k\,(z^2,z^2)_{d-k}} \right),
\end{array}
$$
see \cite{SP}  and \cite{BC1}  for details.
%=======================================================================
\section{Explicit formulas for small  ${\mathbf{d}}$}
%=======================================================================

The formulas of Theorem 3.1 allow the  simplification for some small values  ${\mathbf{d}}.$

\begin{te} Let $s=n$  and   $d_1=d_2=\ldots=d_n=1,$  i.e. ${\mathbf{d}}=(1,1,\ldots,1).$ Then 
$$
\begin{array}{l}
\displaystyle 
\mathcal{PI}_{\mathbf{d}}(z)=\sum_{k=1}^n \frac{(-1)^{n-k}}{(k-1)!} \frac{(n)_{n-k}}{(n-k)! } \frac{d^{k-1}}{dz^{k-1}}\left( \left( \frac{z}{1-z^2} \right)^{2n-k-1} \right),\\
\\
\displaystyle 
\mathcal{PS}_{\mathbf{d}}(z)=\sum_{k=1}^n \frac{(-1)^{n-k}}{(k-1)!} \frac{(n)_{n-k}}{(n-k)! } \frac{d^{k-1}}{dz^{k-1}} \left( \frac{(1+z)z^{2n-k-1}}{(1-z^2)^{2n-k}} \right),
\end{array}
$$
where
$(n)_m:=n(n+1)\cdots (n+m-1),$ $(n)_0:=1$ denotes the shifted factorial.
\end{te}
\begin{proof}
For ${\mathbf{d}}=(1,1,\ldots,1) \in \mathbb{Z}^n$ we have $d^*=1$ and 
$$
f_{{\mathbf{d}}}(tz^{d^*},z)=\frac{1}{\bigl((1-t)(1-t z^2)\bigr)^n}=\frac{A_{0,1}(z)}{1-t}+\cdots+\frac{A_{0,n}(z)}{(1-t)^n}+R(z), \Psi_{1,1}\left(R(z)\right)=0,
$$
where
$$
A_{0,k}=\frac{(-1)^{n-k}}{(n-k)!} \lim_{t \to 1} \frac{\partial^{{n-k}}}{\partial t^{{n-k}}}\left(\frac{1}{(1-t z^2)^n} \right).
$$
By induction we get 
$$
 \lim_{t \to 1}  \frac{\partial^{{m}}}{\partial t^{{m}}}\left(\frac{1}{(1-t z^2)^n} \right)=(n)_m \frac{(z^2)^m}{(1-z^2)^{n+m}}.
$$
Thus,
$$
A_{0,k}=\frac{(-1)^{n-k}(n)_{n-k}}{(n-k)!} \frac{(z^2)^{n-k}}{(1-z^2)^{2n-k}}.
$$

Now, using Theorem 3.1 and the property  $\varphi_1(F(z))=F(z),$ for any $F(z) \in \mathbb{Z}[[z]]$  we have
$$
\begin{array}{l}
\displaystyle \mathcal{PS}_{\mathbf{d}}(z)=\Psi_{1,1}\left( \sum_{k=1}^s \frac{(1+z)\,A_{0,k}}{(1-t)^k} \right)=  \sum_{k=1}^s \Psi_{1,1}\left( \frac{(1+z)\,A_{0,k}}{(1-t)^k} \right) 
\displaystyle =\\
\\
\displaystyle = \sum_{k=1}^n \frac{1}{(m-1)!} \frac{d^{k-1}}{dz^{k-1}} \left( z^{k-1} \varphi_1 ((1+z)\,A_{0,k})\right)=
 \sum_{k=1}^n \frac{1}{(m-1)!} \frac{d^{k-1}}{dz^{k-1}} \left( (1+z)\,z^{k-1} A_{0,k} \right)=\\
\\
\displaystyle =\sum_{k=1}^n \frac{(-1)^{n-k}}{(k-1)!} \frac{(n)_{n-k}}{(n-k)! } \frac{d^{k-1}}{dz^{k-1}} \left( \frac{(1+z)z^{2s-k-1}}{(1-z^2)^{2n-k}} \right).
\end{array}
$$

The case  $ \mathcal{PI}_{{\mathbf{d}}}(z)$ can be considered similarly.
%=================================================================
\end{proof}

\begin{te} Let  $d_1=d_2=\ldots=d_n=2,$ ${\mathbf{d}}=(2,2,\ldots,2)$, then 
$$
\begin{array}{l}
\displaystyle \mathcal{PI}_{\mathbf{d}}(z)=\sum_{k=1}^n \frac{(-1)^{n-k}}{(n-k)(k-1)!}\frac{d^{k-1}}{dz^{k-1}} \left(\sum_{i=0}^{n-k} { n-k  \choose i}  \frac{(n)_i (n)_{n-k-i}(1-z) z^{2n-k-i-1}}{(1- z)^{n+i} (1-z^2)^{2n-k-i}}\right),
\\
\\
\displaystyle \mathcal{PS}_{\mathbf{d}}(z)=\sum_{k=1}^n \frac{(-1)^{n-k}}{(n-k)!(k-1)!}\frac{d^{k-1}}{dz^{k-1}} \left( \sum_{i=0}^{n-k} { n-k  \choose i}  \frac{(n)_i (n)_{n-k-i} z^{2n-k-i-1}}{(1- z)^{n+i} (1-z^2)^{2n-k-i}}\right).
\end{array}
$$

\end{te}
\begin{proof}
It is easy to check  that  in this   case  we have 
$$
f_{{\mathbf{d}}}(tz^2,z)=\frac{1}{\bigl((1-t)(1-t z^2)(1-tz^4)\bigr)^n}.
$$
The decomposition   $f_{{\mathbf{d}}}(tz^2,z)$  into partial fractions yields 
$$
f_{{\mathbf{d}}}(tz^2,z)=\sum_{k=1}^n \left( \frac{A_k(z)}{(1-t)^k} +\frac{B_k(z)}{(1-tz^2)^k} +\frac{C_k(z)}{(1-tz^4)^k}\right),
$$
for some rational functions  $A_k(z), B_k(z), C_k(z).$ Then 
$$
\begin{array}{l}
\displaystyle  \mathcal{PS}_{\mathbf{d}}(z)=\Psi_{1,2}\left((1+z) f_{{\mathbf{d}}}(tz^2,z) \right)=\\
\\
\displaystyle =\sum_{k=1}^n \left( \Psi_{1,2}\left( \frac{(1+z) A_k(z)}{(1-t)^k}\right) +\Psi_{1,2}\left(\frac{(1+z) B_k(z)}{(1-tz^2)^k} \right)+\Psi_{1,2}\left(\frac{(1+z) C_k(z)}{(1-tz^4)^k}\right) \right).
\end{array}
$$
Lemma 3.1 implies that 
$$
\Psi_{1,2}\left(\frac{(1+z) C_k(z)}{(1-tz^4)^k}\right)=0, 
$$
and 
$$
\Psi_{1,2}\left(\frac{(1+z) B_k(z)}{(1-tz^2)^k} \right)=\frac{B_k(0)}{(1-z)^k}, k=1,\ldots, n.
$$
But
$$
B_k(z)=\frac{(-1)^{n-k}}{(n-k)! (z^2)^{n-k}} \lim_{t \to z^{-2}} \frac{\partial^{{n-k}}}{\partial t^{{n-k}}}\left(\frac{1}{(1-t)^n (1-tz^4)^n}\right).
$$
It is easy to see that this parial derivatives has the following form  
$$
\frac{\partial^{{n-k}}}{\partial t^{{n-k}}}\left(\frac{1}{(1-t)^n (1-tz^4)^n}\right)=\frac{\overline B_k(t,z)}{((1-t)(1-tz^4))^{2n-k}}.
$$
for some polynomial $B_k(t,z).$ Moreover,  $ \deg_t(\overline B_k(t,z)) =n-k.$ Then 
$$
\begin{array}{l}
\displaystyle  B_k(z)=\frac{(-1)^{n-k}}{(n-k)! (z^2)^{n-k}} \lim_{t \to z^{-2}}\frac{\overline B_k(t,z)}{((1-t)(1-tz^4))^{2n-k}}=\\
\\
\displaystyle  =\frac{(-1)^{n-k}z^{2n}\overline B_k(1/z^2,z)}{(n-k)! ((z^2-1)(1-tz^4))^{2n-k}}.
\end{array}
$$
It follows that $B_k(z)$ has the factor  $z^{ 2k}$ and then $B_k(0)=0.$ Thus 
$$
\Psi_{1,2}\left(\frac{(1+z) B_k(z)}{(1-tz^2)^k} \right)=0, k=1,\ldots, n.
$$
Therefore
$$
\displaystyle  \mathcal{PS}_{\mathbf{d}}(z)=\sum_{k=1}^n  \Psi_{1,2}\left( \frac{(1+z) A_k(z)}{(1-t)^k}\right)=\sum_{k=1}^n \frac{1}{(k-1)!}\frac{d^{k-1}}{dz^{k-1}} \left(z^{k-1}\varphi_2((1+z)A_k(z))\right).
$$
Let us to calculate  $ A_k(z).$ We have
$$
\begin{array}{l}
\displaystyle A_k(z)=\frac{(-1)^{n-k}}{(n-k)!} \lim_{t\to1}\frac{d^{n-k}}{dt^{n-k}}\left (f_{{\mathbf{d}}}(tz^{2},z) (1-t)^n  \right)=
\\
\\
\displaystyle=\frac{(-1)^{n-k}}{(n-k)!} \lim_{t\to1}\frac{d^{n-k}}{dt^{n-k}}\left ( \frac{1}{(1-tz^2) ^n (1-t z^4)^n}  \right)=\\
\\
\displaystyle =\frac{(-1)^{n-k}}{(n-k)!} \lim_{t\to1} \sum_{i=0}^{n-k} { n-k  \choose i} \left( \frac{1}{(1-tz^2) ^n}  \right)^{(i)}_t  \left( \frac{1}{(1-tz^4) ^n}  \right)^{(n-k-i)}_t=
\\
\\
\displaystyle=\frac{(-1)^{n-k}}{(n-k)!} \lim_{t\to1} \sum_{i=0}^{n-k} { n-k  \choose i} (n)_i (n)_{n-k-i} \frac{z^{2i}}{(1-t z^2)^{n+i}} \frac{z^{4(n-k-i)}}{(1-t z^4)^{2n-k-i}}=\\
\\
\displaystyle =\frac{(-1)^{n-k}}{(n-k)!}  \sum_{i=0}^{n-k} { n-k  \choose i} (n)_i (n)_{n-k-i} \frac{(z^2)^{2(n-k)-i}}{(1- z^2)^{n+i} (1-z^4)^{2n-k-i}}.
\end{array}
$$

Taking into account that  
$\varphi_2(F(z^2))=F(z),$ and $\varphi_2(zF(z^2))=0$
we obtain 
$$
\begin{array}{l}
\displaystyle \varphi_2((1+z)A_k(z) )=\varphi_2(A_k(z))=\\
\\
\displaystyle =\frac{(-1)^{n-k}}{(n-k)!} \sum_{i=0}^{n-k} { n-k  \choose i} (n)_i (n)_{n-k-i} \frac{(z)^{2(n-k)-i}}{(1- z)^{n+i} (1-z^2)^{2n-k-i}}.
\end{array}
$$
Thus,
$$
\begin{array}{l}
\displaystyle  \mathcal{PS}_{\mathbf{d}}(z)=\sum_{k=1}^n  \Psi_{1,2}\left( \frac{(1+z) A_k(z)}{(1-t)^k}\right)=\sum_{k=1}^n \frac{1}{(k-1)!}\frac{d^{k-1}}{dt^{k-1}}\left(z^{k-1} \varphi_2((1+z)A_k(z) \right)=\\
\\
\displaystyle = \sum_{k=1}^n \frac{(-1)^{n-k}}{(n-k)!(k-1)!}\frac{d^{k-1}}{dz^{k-1}} \left( \sum_{i=0}^{n-k} { n-k  \choose i}  \frac{(n)_i (n)_{n-k-i} z^{2n-k-i-1}}{(1- z)^{n+i} (1-z^2)^{2n-k-i}}\right).
\end{array}
$$

By replacing the factor $1+z$  with $1-z^2$ in $\mathcal{PS}_{{\mathbf{d}}}(z)$  and taking  into account that  $$\varphi_2((1-z^2) A_k(z))=(1-z) \varphi_2(A_k(z)),$$

 get the Poincar\'e series  $\mathcal{PI}_{{\mathbf{d}}}(z).$
\end{proof}

\section{Examples}  

 For direct computations of the function $\varphi_n$ we use the following technical lemma, see \cite{BC1}:

\begin{lm} Let   $R(z)$ be some polynomial of $z.$ Then
$$
\varphi_n\left(\frac{R(z)}{(1-z^{k_1})(1-z^{k_2})\cdots(1-z^{k_m})} \right)= 
\frac{\varphi_n\bigr(R(z)Q_n(z^{k_1})Q_n(z^{k_2})Q_n(z^{k_m})\bigr)}{(1-z^{k_1})(1-z^{k_2})\cdots(1-z^{k_m})},
$$
here  $Q_n(z)=1+z+z^2+\ldots+z^{n-1},$  and  $k_i$ are natural numbers.
\end{lm}
As example, let us calculate the Poincar\'e series   $\mathcal{PD}_{(1,2,3)}.$ We have $d^*=3$  and 
$$
f_{(1,2,3)}(t,z)={\frac {1}{ \left( 1-t{z}^{4} \right) ^{2} \left( 1-t{z}^{2}
 \right) ^{2} \left( 1-t{z}^{5} \right)  \left( 1-t{z}^{3} \right) 
 \left( 1-tz \right)  \left( 1-t{z}^{6} \right)  \left( 1-t \right) }}.
$$
The decomposition  $f_{(1,2,3)}(t,z)$ into partial fractions yelds:
$$
\begin{array}{l}
\displaystyle f_{(1,2,3)}(t,z)={\frac {A_{{0,1}}(z)}{1-t}}+{\frac {A_{{1,1}}(z)}{1-tz}}+{\frac {A_{{2,1}}(z)}{
1-t{z}^{2}}}+{\frac {A_{{2,2}}(z)}{ \left( 1-t{z}^{2} \right) ^{2}}}+{
\frac {A_{{3,1}}(z)}{1-t{z}^{3}}}+{\frac {A_{{4,1}}(z)}{1-t{z}^{4}}}+
\\
\\
\displaystyle +
{\frac 
{A_{{4,2}}(z)}{ \left( 1-t{z}^{4} \right) ^{2}}}+
{\frac {A_{{5,1}}(z)}{1-t{z
}^{5}}}+{\frac {A_{{6,1}}(z)}{1-t{z}^{6}}}.
\end{array}
$$
By using Lemma  3.1 we have 
$$
\begin{array}{l}
\displaystyle \mathcal{PD}_{(1,2,3)}(z)=\Psi_{1,3}\left( (1+z) f_{(1,2,3)}(t,z) \right)=\\
\\
\displaystyle  =\Psi_{1,3}\left({\frac {(1+z)A_{{0,1}}(z)}{1-t}}\right) +\Psi_{1,3}\left({\frac {(1+z)A_{{1,1}}(z)}{1-tz}}\right)+\\
\\
\displaystyle +\Psi_{1,3}\left({\frac {(1+z) A_{{2,1}}}{
1-t{z}^{2}}}\right)+\Psi_{1,3}\left({\frac {(1+z) A_{{2,2}}(z)}{ \left( 1-t{z}^{2} \right) ^{2}}}\right)+\Psi_{1,3}\left({
\frac {(1+z) A_{{3,1}}(z)}{1-t{z}^{3}}}\right)=\\
\\
\displaystyle =\varphi_{3}\left((1+z)A_{0,1}(z)\right) +\varphi_{2}\left((1+z)A_{1,1}(z)\right)+\varphi_{1}\left((1+z)A_{{2,1}}(z) \right)+\\
\\
\displaystyle + \left(z \varphi_{1}\left((1+z)A_{{2,2}}(z) \right)\right)'_z+A_{{3,1}}(0).
\end{array}
$$
Now
$$
\begin{array}{l}
\displaystyle  A_{0,1}(z)=\lim_{t\to 1}\left( f_{(1,2,3)}(t,z)(1-t)\right)=\\
\\
\displaystyle  ={\frac {1}{ \left( 1-{z}^{4} \right) ^{2} \left( 1-{z}^{2} \right) ^
{2} \left( 1-{z}^{5} \right)  \left( 1-{z}^{3} \right)  \left( 1-z
 \right)  \left( 1-{z}^{6} \right) }}.
\end{array}
$$
and 
$$
\begin{array}{l}
\displaystyle \varphi_3((1+z)A_{0,1}(z))=\\
\\
\displaystyle={\frac {2\,{z}^{11}+7\,{z}^{10}+14\,{z}^{9}+29\,{z}^{8}+34\,{z}^{7}+42
\,{z}^{6}+42\,{z}^{5}+33\,{z}^{4}+21\,{z}^{3}+14\,{z}^{2}+4\,z+1}{
 \left( 1-{z}^{5} \right)  \left( 1-z \right) ^{3} \left( 1-{z}^{4} \right) ^{2} \left( 1-{z}^{2}\right) ^{2}}}.
\end{array}
$$
As above we obtain

$$
\begin{array}{l}
\displaystyle  A_{1,1}(z)=\lim_{t\to z^{-1}}\left( f_{(1,2,3)}(t,z)(1-t z)\right)=\\
\\
\displaystyle  ={\frac {z}{ \left( 1-{z}^{3} \right) ^{2} \left( 1-z \right) ^{2}
 \left( 1-{z}^{4} \right)  \left( 1-{z}^{2} \right)  \left( 1-{z}^{5}
 \right)(z-1) }},
\end{array}
$$
$$
\begin{array}{l}
\displaystyle \varphi_2((1+z)A_{1,1}(z))=-{\frac {z \left( 4+13\,{z}^{2}+6\,z+6\,{z}^{6}+{z}^{7}+13\,{z}^{4}+9
\,{z}^{5}+12\,{z}^{3} \right) }{ \left( 1-z^2 \right)  \left( 1-z^5 \right)  \left(1-z^3 \right) ^{2} \left( 1-z
 \right) ^{4}}}.
\end{array}
$$
$$
\begin{array}{l}
\displaystyle  A_{2,1}(z)=-\frac{1}{z^2}\lim_{t\to z^{-2}}\left( f_{(1,2,3)}(t,z)(1-t z^2)^2\right)'_t=\\
\\
\displaystyle  =-{\frac {{z}^{3} \left( 5\,{z}^{6}+5\,{z}^{5}+6\,{z}^{4}+2\,{z}^{3}-{z
}^{2}-2\,z-2 \right) }{ \left( 1-{z}^{4} \right) ^{2} \left( 1-z
 \right)  \left( 1-{z}^{3} \right) ^{2} \left( 1-z^2 \right) ^{3}}},\\
\\
\displaystyle \varphi_1((1+z) A_{2,1}(z))=-{\frac {{z}^{3} \left( 5\,{z}^{6}+5\,{z}^{5}+6\,{z}^{4}+2\,{z}^{3}-{z
}^{2}-2\,z-2 \right) }{ \left( 1-{z}^{4} \right) ^{2} \left( 1-z
 \right)^2  \left( 1-{z}^{3} \right) ^{2} \left( 1-z^2 \right) ^{2}}}.\\
\\

\displaystyle  A_{2,2}(z)=\lim_{t\to z^{-2}}\left( f_{(1,2,3)}(t,z)(1-t z)^2\right) ={\frac {{z}^{3}}{ \left( 1-{z}^{4} \right)  \left( 1-z \right) ^{3}
 \left( 1-{z}^{3} \right)  \left( 1-{z}^{2} \right) ^{2}}},\\
\\
\displaystyle \left(z \varphi_{1}\left((1+z) A_{{2,2}}(z) \right)\right)'_z=\left(z (1+z)  A_{{2,2}}(z) \right)'_z=\\
\\
\displaystyle ={\frac {{z}^{3} \left( 10\,{z}^{6}+13\,{z}^{5}+20\,{z}^{4}+16\,{z}^{3}
+14\,{z}^{2}+7\,z+4 \right) }{ \left( 1-{z}^{2} \right) ^{2} \left( 1-{z}^{3} \right) ^{2} \left( 1-z \right) ^{2} \left(1-z^4\right)^2  }}.
\end{array}
$$
At last 
$$
\begin{array}{l}
\displaystyle  A_{3,1}(z)=\lim_{t\to z^{-3}}\left( f_{(1,2,3)}(t,z)(1-t z^3)\right) ={\frac {{z}^{7}}{ \left( 1-{z}^{3} \right) ^{2} \left( 1-z \right) ^
{5} \left( 1-{z}^{2} \right) }}.
\end{array}
$$
Thus $ A_{3,1}(0)=0.$

After summation and simplification we obtain the explicit expression for the Poincar\'e series 
$$
\begin{array}{l}
\displaystyle  \mathcal{PD}_{(1,2,3)}(z)={\frac {p_{(1,2,3)}(z)}{\left( 1-{z}^{4} \right) ^{2} \left( 1-z \right) ^{2} \left( 1-{z}^{
2} \right)  \left( 1-{z}^{3} \right) ^{2} \left( 1-{z}^{5}\right)
}},
\end{array}
$$
where 
$$
\begin{array}{l}
\displaystyle p_{1,2,3}(z)={z}^{14}+{z}^{13}+6\,{z}^{12}+12\,{z}^{11}+20\,{z}^{10}+29\,{z
}^{9}+35\,{z}^{8}+39\,{z}^{7}+35\,{z}^{6}+29\,{z}^{5}+\\
\\
\displaystyle +20\,{z}^{4}+12\,
{z}^{3}+6\,{z}^{2}+z+1.
\end{array}
$$
%===================================================

The following  Poincar\'e series obtained 
by using the explicit formulas of Theorem 3.2 and Theorem 3.3 

$$
 \mathcal{PD}_{(1,1)}(z)={\frac {1}{ \left( 1-z \right) ^{2} \left( 1-{z}^{2} \right) }},  \mathcal{PD}_{(1,1,1)}(z)={\frac {1-{z}^{3}}{ \left( 1-z \right) ^{3} \left( 1-{z}^{2}
 \right) ^{3}}},
$$
$$
 \mathcal{PD}_{(1,1,1,1)}(z)={\frac {{z}^{4}+2\,{z}^{3}+4\,{z}^{2}+2\,z+1}{ \left( 1-z \right) ^{2
} \left( 1-{z}^{2} \right) ^{5}}},
$$
$$
 \mathcal{PD}_{(1,1,1,1,1)}(z)={\frac {{z}^{6}+3\,{z}^{5}+9\,{z}^{4}+9\,{z}^{3}+9\,{z}^{2}+3\,z+1}{
 \left( 1-z \right) ^{2} \left( 1-{z}^{2} \right) ^{7}}},
$$
$$
 \mathcal{PD}_{(1,1,1,1,1,1)}(z)={\frac {{z}^{8}+4\,{z}^{7}+16\,{z}^{6}+24\,{z}^{5}+36\,{z}^{4}+24\,{z
}^{3}+16\,{z}^{2}+4\,z+1}{ \left( 1-z \right) ^{2} \left( 1-{z}^{2}
 \right) ^{9}}}
$$
$$
 \mathcal{PD}_{(1,1,1,1,1,1,1)}(z)={\frac {p_7(z)}{ \left( 1-z
 \right) ^{2} \left( 1-{z}^{2}\right) ^{11}}}
$$
$$
p_7(z)={z}^{10}+5\,{z}^{9}+25\,{z}^{8}+50\,{z}^{7}+100\,{z}^{6}+100
\,{z}^{5}+100\,{z}^{4}+50\,{z}^{3}+25\,{z}^{2}+5\,z+1.
$$

$$
 \mathcal{PD}_{(2,2,2)}(z)={\frac {1+4\,{z}^{2}+{z}^{4}}{ \left( 1-z \right) ^{3} \left( 1-z^2
 \right) ^{5}}},  \mathcal{PD}_{(2,2,2,2)}(z)={\frac {1+9\,{z}^{2}+9\,{z}^{4}+{z}^{6}}{ \left( 1-z \right) ^{4} \left( 1-z^2
 \right) ^{7}}}
$$
$$
\begin{array}{l}
 \displaystyle \mathcal{PD}_{(2,2,2,2,2)}(z)={\frac {1+16\,{z}^{2}+36\,{z}^{4}+16\,{z}^{6}+{z}^{8}}{\left( 1-z \right) ^{5} \left( 1-z^2
 \right) ^{9}}},\\
\\
 \displaystyle \mathcal{PD}_{(2,2,2,2,2,2)}(z)={\frac {{z}^{10}+25\,{z}^{8}+100\,{z}^{6}+100\,{z}^{4}+25\,{z}^{2}+1}
{ \left( 1-z \right) ^{6} \left( 1-{z}^{2} \right) ^{11}}},\\
\\
 \displaystyle \mathcal{PD}_{(2,2,2,2,2,2,2)}(z)={\frac {{z}^{12}+36\,{z}^{10}+225\,{z}^{8}+400\,{z}^{6}+225\,{z}^{4}+
36\,{z}^{2}+1}{ \left( z-1 \right) ^{7} \left( 1-{z}^{2} \right) ^{13
}}}.
\end{array}
$$

By using Maple we computed the Poincar\'e series up to $n=30.$ The case $n=2$ agrees to  the results of the paper \cite{BC2}.

\end{document}